\documentclass[DIV=12,10pt]{scrartcl}

\usepackage[breaklinks=true]{hyperref}
\usepackage{xcolor}
\usepackage{amsmath,amsthm,amsfonts,amssymb}
\usepackage{mathdots}
\usepackage{mathtools}
\usepackage{enumitem}
\usepackage{longtable}
\usepackage{caption}
\usepackage{afterpage}
\usepackage{stackengine}
\usepackage[capitalize]{cleveref}

\definecolor{darkblue}{rgb}{0.0, 0.0, 0.6}

\hypersetup{
  pdfauthor={Andrei Smolensky},
  pdftitle={Cayley graph diameters for fixed cycle types are eventually quasipolynomial},  
  colorlinks=true,
  urlcolor=darkblue,
  linkcolor=darkblue,
  citecolor=darkblue}

\usepackage[backgroundcolor=white,linecolor=lightgray]{todonotes}

\newtheorem{theorem}{Theorem}
\newtheorem*{theorem*}{Theorem}
\newtheorem{lemma}{Lemma}
\newtheorem{corollary}{Corollary}
\newtheorem*{corollary*}{Corollary}
\theoremstyle{definition}
\newtheorem{definition}{Definition}
\newtheorem*{definition*}{Definition}
\newtheorem{remark}{Remark}
\newtheorem*{remark*}{Remark}

\crefalias{equations}{equation}
\crefname{equations}{equations}{equations}
\creflabelformat{equations}{#2\textup{(#1)}#3} 

\newcommand{\cvec}{\mathbf{c}}
\newcommand{\mvec}{\mathbf{m}}
\newcommand{\evec}{\mathbf{e}}

\newcommand{\dict}{\mathcal{A}}

\newcommand{\cc}{\mathcal{C}}

\title{Cayley graph diameters for fixed cycle types are eventually quasipolynomial}
\author{Andrei Smolensky\thanks{Neapolis University Pafos, Cyprus, \texttt{andrei.smolensky@gmail.com}}}

\begin{document}
\maketitle

\begin{abstract}
Given a cycle type, the corresponding conjugacy class of $S_n$ generates either $S_n$ or $A_n$ for sufficiently large $n$. We prove that the sequence of diameters of Cayley graphs is eventually polynomial on residue classes for any fixed cycle type. This result is also extended to finite unions of conjugacy classes.
\end{abstract}


\section*{Introduction}

Fix a non-empty multiset of integers $\lambda = (\lambda_1,\ldots,\lambda_h)$ with every $\lambda_i\geqslant 2$, and consider the set $\cc_{\lambda,n}$ of all $n$-letter permutations of cycle type $\lambda$. In cycle notation we usually omit cycles of length $1$.

The sequence $D_n$ of diameters of Cayley graphs of the groups generated by $\cc_{\lambda,n}$ has gained some attention recently \cite{MO512777}. It was conjectured to be an instance of the \emph{quasipolynomiality phenomenon:} that this sequence is \emph{eventually polynomial on residue classes,} i.e. that there exist a starting point $N$ and a period $P$ such that $D_n$ restricted to $n\in N+r+P\mathbb{N}$ is a polynomial for each $r=0,\ldots,P-1$.

Specifically for conjugacy classes in $A_n$, this sequence was previously shown by M.~Liebeck and A.~Shalev \cite[Lemma~3.2]{LiebeckShalevDiametersSharpBounds} to be $O(n)$, so these polynomials must be of degree at most $1$.

\begin{definition}
An integer sequence $a_n$ is \emph{eventually affine PORC} (with a common positive slope) if there exist positive integers $P,q,N$ such that $a_{n+P} = a_n+q$ whenever $n\geqslant N$. Equivalently, $a_n = (q/P)n + c_{n \bmod P}$ for some rational constants $c_0,\ldots,c_{P-1}$ and $n\geqslant N$.
\end{definition}

The main result of the present paper is the following:
\begin{theorem}
	\label{thm:diameters-affine-porc}
	For every fixed non-trivial cycle type $\lambda$, the diameters $D_n$ are eventually affine PORC.
\end{theorem}
This theorem can also be extended to finite unions of conjugacy classes (see~\cref{sec:conjugacy-classes-unions} for the proof and further discussion). Both results are supplied with a Lean formalization\footnote{Formalization is available at \href{https://github.com/special-linear/conjugacy_diameters_porc_lean}{\texttt{github.com/special-linear/conjugacy\_diameters\_porc\_lean}}.}.

The property of diameters being PORC has been previously established for many explicit generating sets, and conjectured for many more based on the computational evidence \cite{CayleyPyGrowth}. To the best of our knowledge, this is the first case where quasipolynomiality is proven without establishing any sort of explicit formula for the diameter.

The diameters have been previously calculated for some specific cycle types and for some unions:
\begin{itemize}
\item For $\lambda=(2)$, the diameter with respect to all transpositions is trivially $n-1$.
\item More generally, for $\lambda=(3)$ and $\lambda=(4)$, it follows from the results of M.~Herzog and K.~B.~Reid \cite{HerzogReidCycles} that the diameters are $\lfloor n/2\rfloor$ and $\max(3, \lceil n/2\rceil)$, respectively.
\item The results of \cite{ProductsOfCycleClassesI,ProductsOfCycleClassesII} imply that the diameters for odd cycles of length $\ell>3$ are
\[
	\begin{cases}
		\lceil (n-1)/\nu_\ell\rceil + [n\equiv \nu_\ell+1 \pmod{2\nu_\ell}], & \text{if } \ell\equiv 0 \pmod{3}, \\
		\lceil (n-2)/\nu_\ell\rceil, & \text{if } \ell\equiv 1\pmod{3}, \\
		\lceil (n-1)/\nu_\ell\rceil + [n\equiv \nu_\ell+1 \pmod{4\nu_\ell}], & \text{if } \ell\equiv 2 \pmod{3},
	\end{cases}
	\quad\text{where}\quad
	\nu_\ell = \left\lfloor \frac{2\ell}{3} \right\rfloor.
\]
For even cycles, it follows from the authors' results that the diameters are (for $\ell\geqslant 4$ and $n\geqslant\ell+2$)
\[
	2\left\lceil \frac{n-1-[\ell\equiv 1\pmod{3}]}{2\nu_\ell} \right\rceil + \varepsilon_\ell(n),
	\quad\text{where}\quad
	\varepsilon_\ell(n) \in \{0,1\}.
\]
These results do no establish whether the sequence $\varepsilon_\ell(n)$ is periodic.
\item For conjugacy classes of elements of order $2$, that is, of cycle type $2^s$, the diameter formula
\[
	\left\lceil \frac{n-2}{s} \right\rceil + [\text{$s$ odd}]
\]
for $n\geqslant 4s$ is established in \cite{PhongpattanacharoenSiemons}.
\item In \cite{KMCycles}, the diameters with respect to all products of $k$ $\ell$-cycles are calculated for $\ell=2,3,4$, thus effectively calculating the diameters relative to some unions of conjugacy classes:
\[
	\left\lceil \frac{n-2}{k} + [\text{$k$ odd}] \right\rceil\ \text{for $\ell=2$}, \quad
	\left\lceil \frac{\lfloor n/2\rfloor}{k} \right\rceil\ \text{for $\ell=3$}, \quad
	\left\lceil \frac{n-2}{2k} + [\text{$k$ odd}] \right\rceil\ \text{for $\ell=4$}.
\]
\item Diameters relative to all permutations supported on at most $k$ letters have been shown to be $\lceil (n-1)/(k-1) \rceil$ in \cite[Lemma~14]{Mastermind}.
\item In \cite{CicirelloCycleMutation} the diameters with respect to all cycles of length at most $\ell$ are discussed, and it is shown that for $\ell=3,4$ the diameters are $\lfloor n/2\rfloor$.
\end{itemize}

\section{Proof of \cref{thm:diameters-affine-porc}}

Let us now introduce the notation used in the paper.

For each $\ell\geqslant 2$, let $m_\ell = \#\{i \mid \lambda_i=\ell\}$ denote the number of cycles of length $\ell$, and set
\begin{align*}
	& J = \{ \ell \mid m_\ell > 0 \} && \text{the set of non-trivial cycle lengths present in $\cc_{\lambda,n}$,} \\
	& b = \sum_{i=1}^h \lambda_i =\sum_{\ell\in J} \ell m_\ell && \text{the generator support size,} \\
	& h = \sum_{\ell\in J} m_\ell && \text{the number of non-trivial cycles,} \\
	& d = b-h = \sum_{\ell\in J} (\ell-1) m_\ell && \text{the length of the generators relative to all transpositions.}
\end{align*}

Denote $G_n = \langle \cc_{\lambda,n} \rangle$. For $n\geqslant b+1$, $G_n = S_n$ if $d$ is odd, and $A_n$ if $d$ is even.

Define $L_n(g)$ as the length of the shortest expression of $g$ as a word in elements of $\cc_{\lambda,n}$. The diameter is defined as $D_n = \max_{g\in G_n} L_n(g)$.

The main idea of the proof is that, for sufficiently large $n$, a permutation has a long cycle (of length at least $b$) or a repeated short cycle, and both can be either shortened or removed entirely using a bounded number of generators. This can be done using a finite ``dictionary'', which reduces the problem of numerical bounds to finite rational optimization.

\begin{definition}
A \emph{constituent word} is a finite ordered list $w=(a_1,\ldots,a_K)$, where each $a_i$ is a cycle of some length in $J$. Its \emph{value} is the product $a_1\cdots a_K$. Denote by $c_\ell(w)$ the number of $\ell$-cycles in $w$, and denote by $\cvec(w) = \big( c_\ell(w) \big)_{\ell\in J}$ its \emph{resource vector}. Denote
\[
\delta(c) = \sum_{\ell\in J} (\ell-1) c_\ell,
\]
the length of the obvious decomposition of the value into transpositions coming from the decomposition of each cycle factor (the product can of course have a shorter decomposition).
\end{definition}

If a word of $q$ generators of type $\lambda$ is expanded into its individual cycle factors, it becomes a constituent word with resource vector $q\mvec$, where $\mvec = (m_\ell)_{\ell\in J}$.

Fix $r\geqslant 2$ and consider the words with value of cycle type $r^s$. When this is the case, we say that $w$ has \emph{target-cycle count} $s$. Denote also $M=r+1$.

\begin{lemma}
	\label{lemma:baseline_word}
	For each $\ell\in J$ there exists a constituent word with exactly $M$ factors, all $\ell$-cycles, whose value has type $r^{\ell-1}$ on $r(\ell-1)+1$ letters. Its resource vector is $M\evec_\ell$, where $\evec_\ell$ is the $\ell$-th coordinate unit vector, and its target-cycle count satisfies
	\begin{equation}
		\label{eq:baseline-word-target-cycle-count}
		M(\ell-1) = \delta(M\evec_\ell).
	\end{equation}
\end{lemma}
\begin{proof}
Denote $u=\ell-1$ and $v=ru+1$, and set
\[
B_j = (ju+1,ju+2,\ldots,(j+1)u+1) \quad \text{for} \quad 0\leqslant j <r,
\]
so that $B_0=(1,2,\ldots,\ell)$, $B_{r-1} = (v-\ell+1,\ldots,ru+1)$, and
\[
T = (1,1+r,1+2r,\ldots,1+ur).
\]
The permutations $B_j$ and $T$ are all $\ell$-cycles. A straightforward calculation shows that
\begin{gather*}
	B_0B_1\cdots B_{r-1} = (1,2,\ldots,v), \\
	T^{-1}B_0B_1\cdots B_{r-1} = (1,2,\ldots,r)(r+1,\ldots,2r)\cdots((u-1)r+1,\ldots,ur). \qedhere
\end{gather*}
\end{proof}
The words provided by the above lemma will be called \emph{baseline words.}

Given a constituent word, define its \emph{support graph} by taking the letters appearing in its factors for vertices, and joining letters appearing together in some factor. Each factor is thus contained in a connected component. Moreover, the value of type $r^s$ can be restricted to any component, producing a permutation of type $r^{s'}$.

\begin{definition}
	\label{def:ecs}
Given a constituent word with resource vector $\cvec$ and target-cycle count $s$, denote $E(\cvec,s) = \delta(\cvec)-Ms+2$.
\end{definition}
\begin{lemma}
	\label{lemma:connected-word-defect}
	Suppose a constituent word has connected support graph, resource vector $\cvec\neq 0$ and value of type $r^s$ on $v=rs+f$ vertices of its support graph. Then
	\[
		E(\cvec,s) \geqslant 2f \geqslant 0.
	\]
\end{lemma}
\begin{proof}
Let $V$ be the union of the supports of the constituent factors, so that $|V|=v=rs+f$. Write the word value as $g=a_1\cdots a_K$. The word $g^{-1}\cdot a_1\cdots a_K = 1$ can be expanded into a transposition word 
\[
	t_1\cdots t_L = 1, \quad L=\delta(\cvec)+(r-1)s.
\]
Indeed, $g^{-1}$ is a product of $s$ disjoint $r$-cycles.

Consider the partial products $P_0=1$, $P_k = P_{k-1}t_k = t_1\cdots t_k$. Consider also the sequence $H_k$ of graphs on $V$ with edges corresponding to the transpositions $t_1,\ldots,t_k$ (in particular, $H_0$ is edgeless). Each $P_k$ preserves every connected component of $H_k$, and each cycle of $P_k$ is contained within a component of $H_k$.

The number of graph components starts at $v$ and ends at $1$. Indeed, the transpositions coming from the factors $a_1,\ldots,a_K$ already connect all vertices. Adding an edge decreases this number by $1$ precisely when its endpoints lie in different components. Thus there are $v-1$ steps that merge all graph components. At each such step $k$, the endpoints lie in different components of $H_{k-1}$, hence in different cycles of $P_{k-1}$.

In general, when a permutation is multiplied by a transposition, it either joins two of its cycles into one, or splits one cycle into two, depending on what cycles the support of this transposition intersects. For the identity word above and its partial products, the number of cycles starts at $v$ (all points are fixed) and returns to $v$. Hence there are equally many joins and splits. Thus
\[
L = \#\mathrm{joins} + \#\mathrm{splits} = 2\cdot\#\mathrm{joins} \geqslant 2(v-1).
\]
The last inequality holds because the merging of components in $H_k$ happens when two cycles of $P_{k-1}$ are joined.
Substituting $v=rs+f$ yields $\delta(\cvec)+(r-1)s \geqslant 2(rs+f-1)$, and the claimed inequality follows by rearranging.
\end{proof}

We shall use the following form of Dickson's lemma:
\begin{lemma}
	\label{lemma:dickson}
	Let $\mathbb{N}^k$ be ordered coordinate-wise: $u\leqslant v$ if and only if $u_i\leqslant v_i$ for all $i$. Then every non-empty $X\subseteq \mathbb{N}^k$ has finitely many minimal elements, and every $x\in X$ lies above one of them (i.e. for any $x\in X$ there exists a minimal element $y$ such that $y\leqslant x$).
\end{lemma}

For each residue vector $\rho \in(\mathbb{Z}/M\mathbb{Z})^J$, let $X_\rho$ consist of pairs $\big( \cvec, E(\cvec,s) \big) \in \mathbb{N}^{|J|+1}$ arising from non-empty connected constituent words with value of type $r^s$ with any $s$ and with $\cvec \equiv \rho \pmod{M}$. By \cref{lemma:connected-word-defect} they are indeed elements of $\mathbb{N}^{|J|+1}$.

For each coordinate-wise minimal pair of each non-empty $X_\rho$ select one word that realizes it. There are finitely many choices to be made. Form the set $\dict_r$ from these chosen words along with the baseline words provided by \cref{lemma:baseline_word}. This set will serve as a \emph{dictionary}.

We denote, for a word $A\in\dict_r$, its resource vector by $\cvec^A$, and its target-cycle count by $s_A$.

\begin{lemma}
	\label{lemma:finite-domination}
	For any constituent word with value of type $r^s$ there exists a union of dictionary words from $\dict_r$ with exactly the same resource vector and with at least $s$ $r$-cycles in its value, by possibly using extra letters.
\end{lemma}
\begin{proof}
Decompose the original word into non-empty connected components. For the restriction on a single component, write $\cvec, s, E$ as prescribed in \cref{def:ecs}. In the same residue class $\cvec \bmod M$, select a realized dictionary pair $(\cvec^0,E^0)$ with $\cvec^0\leqslant\cvec$ and $E^0\leqslant E$. Such a pair exists by \cref{lemma:dickson}. Since $\cvec-\cvec^0$ is coordinate-wise non-negative and divisible by $M$, the integers
\[
k_\ell = \frac{c_\ell-c^0_\ell}{M}
\]
are non-negative. Take the dictionary word $w^0$ representing $(\cvec^0,E^0)$ and $k_\ell$ baseline words (with $M$ constituent factors, all of length $\ell$), for each $\ell$, on disjoint sets of letters. The total resource vector of this union is
\begin{equation}
	\label{eq:finite-domination:resource-vector}
	\cvec^0 + \sum_{\ell\in J} k_\ell M\evec_\ell = \cvec.
\end{equation}
Indeed, $c^0_\ell + k_\ell M = c^0_\ell + (c_\ell-c^0_\ell) = c_\ell$. If $s^0$ is the target-cycle count of $w^0$, then \cref{eq:baseline-word-target-cycle-count} gives the new count
\[
s' = s^0 + \sum_{\ell\in J} k_\ell(\ell-1) = \frac{\delta(\cvec)+2-E^0}{M}.
\]
The last equality comes from a combination of the following observations. First, $E^0=\delta(\cvec^0)-Ms^0+2$ by definition, so $Ms^0 = \delta(\cvec^0)+2-E^0$. Second, applying the linear function $\delta$ to \cref{eq:finite-domination:resource-vector} gives
\begin{gather*}
	\delta(c) = \delta(\cvec^0) + M\sum_{\ell\in J} k_\ell (\ell-1). \\
	\intertext{Combining them, one gets}
	Ms' = Ms^0 + M\sum_{\ell\in J} k_\ell (\ell-1) = \delta(\cvec^0)+2-E^0 + \delta(\cvec) - \delta(\cvec^0) = \delta(\cvec)+2 - E^0.
\end{gather*}
One gets from \cref{def:ecs} that $s = (\delta(\cvec)+2-E)/M$, hence $s'-s = (E-E^0)/M \geqslant 0$.

It remains to concatenate the words constructed for every connected component. The resource vectors and target-cycle counts add, because the words are realized on disjoint sets.
\end{proof}

Define
\begin{equation}
	\label{eq:gamma-r}
	\Gamma_r = \max\left\{ \sum_{A\in\dict_r} s_A x_A\ \middle|\ x_A\geqslant 0,\ \sum_{A\in\dict_r} \cvec^A x_A = \mvec \right\}.
\end{equation}
The variables $x_A$ are non-negative real weights. The above quantity measures the greatest possible average number of $r$-cycles produced per generator of cycle type $\lambda$.

\begin{remark}
	\label{rmk:gamma-r-positive-rational}
	$\Gamma_r$ is well-defined, positive, and rational.
\end{remark}
\begin{proof}
The feasible set in \cref{eq:gamma-r} is non-empty: assign weight $m_\ell/M$ to each baseline word $A$ (with $\cvec^A = M\evec_\ell$), and $0$ to other words, then the objective becomes
\[
	\sum_{\ell\in J} (\ell-1) \frac{m_\ell}{M} = \frac{d}{M} > 0.
\]
Every $\cvec^A\neq 0$ has coordinate sum at least $1$, while the right hand side of the linear constraint has coordinate sum $h$. Therefore any feasible point satisfies $\sum_A x_A\leqslant h$. The feasible set is closed, bounded and defined by rational equalities and inequalities, hence the conclusion.
\end{proof}

\begin{lemma}\leavevmode
	\label{lemma:sharp-bound}
	\begin{enumerate}
		\item Every factorization of a permutation of cycle type $r^s$ into $q\geqslant 1$ generators from $\cc_{\lambda,n}$, in any ambient degree $n$, satisfies $s\leqslant q\Gamma_r$.
		\item There exist integers $K_r,t_r,N_r>0$ and a factorization on $N_r$ letters of a permutation of cycle type $r^{K_r}$ into $t_r$ generators of type $\lambda$ with $\frac{t_r}{K_r} = \frac{1}{\Gamma_r}$.
	\end{enumerate}
\end{lemma}
\begin{proof}[Proof of the lower bound]
Expand a $q$-generator factorization of a permutation of type $r^s$ into constituent cycles. Its resource vector is $q\mvec$. \cref{lemma:finite-domination} provides a collection of dictionary words (with multiplicities $z_A$) such that
\[
	\sum_{A\in\dict_r} \cvec^A z_A = q\mvec, \quad \sum_{A\in\dict_r} s_A z_A \geqslant s.
\]
Set $x_A = z_A/q$, so that they provide a feasible point in \cref{eq:gamma-r}. Then
\[
	\frac{s}{q} \leqslant \sum_{A\in\dict_r} s_A x_A \leqslant \Gamma_r.
\]
\emph{Proof of finite attainment.} Choose a rational optimal solution $x^*$. Let $Q\geqslant 1$ clear all its denominators, and set $z_A = Q x^*_A\in\mathbb{Z}_{\geqslant 0}$. Concatenate $z_A$ copies of each dictionary word (on disjoint letters). The resulting constituent word $w=(a_1,\ldots,a_K)$ satisfies
\[
	\cvec(w) = Q\mvec, \quad K=Qh, \quad \text{value of $w$ has type $r^S$, where $S=\sum_A s_A z_A = Q\Gamma_r \in\mathbb{Z}_{>0}$.}
\]

The constituent factors in this word may overlap, so we shall now construct a decomposition for a new permutation of type $r^{KS}$, which will be decomposed into $KQ$ generators of type $\lambda$.

Let $V$ be the finite set of letters appearing in $w$. Take $K$ mutually disjoint copies $V_0,\ldots,V_{K-1}$ of $V$. On $V_p$, place a relabelled copy of the cyclically shifted word $w_p = (a_{p+1},\ldots,a_K,a_1,\ldots,a_p)$, $p=0,\ldots,K-1$. Schematically, the arrangement is
\begin{center}
	\begin{tabular}{c|cccc}
		& pos. $1$ & pos. $2$ & $\ldots$ & pos. $K$ \\\hline
		$V_0$ & $a_1$ & $a_2$ & $\ldots$ & $a_K$ \\
		$V_1$ & $a_2$ & $a_3$ & $\ldots$ & $a_1$ \\
		$\vdots$ & $\vdots$ & $\vdots$ && $\vdots$ \\
		$V_{K-1}$ & $a_K$ & $a_1$ & $\ldots$ & $a_{K-1}$
	\end{tabular}
\end{center}
where every entry is relabelled onto the set indicated by its row. Every row has value of type $r^S$, since cyclic shifts of a word have conjugate values

Now consider one column. It contains exactly one relabelled copy of each factor $a_1,\ldots,a_K$, so its resource vector is $Q\mvec$, and these factors have disjoint supports. For each column separately, partition its $K$ entries into $Q$ groups, each containing $m_\ell$ cycles of length $\ell$ for every $\ell$. Since they act on disjoint supports, the product over each group is a permutation of cycle type $\lambda$. Thus the product over each column can be presented as a product of type $\lambda$ generators.

Denote the generators obtained from column $j$ by $f_{j,1},\ldots,f_{j,Q}$, and concatenate the columns in order $1,\ldots,K$. Define
\[
	g = (f_{1,1}\cdots f_{1,Q})(f_{2,1}\cdots f_{2,Q})\cdots(f_{K,1}\cdots f_{K,Q}).
\]
Expanding these generators into constituent factors $a_i$, each row's factors appear in their original order. Factors from different words commute, so we can collect each row's factors together without changing their order or the product. Thus $g$ equals the product of $K$ values of words $w_0,\ldots,w_{K-1}$, which produces a permutation of cycle type $r^{KS}$.

We thus get the claimed result with $t_r=KQ$ and $K_r=KS$, because $\frac{t_r}{K_r} = \frac{KQ}{KS} = \frac{1}{\Gamma_r}$. The number of letters used is $N_r=K|V|$.
\end{proof}

Denote
\begin{equation}
	\label{eq:common-slope}
	\alpha_\lambda = \max\left( \frac{1}{d}, \max_{2\leqslant r < b} \frac{1}{r\Gamma_r} \right).
\end{equation}
Here the inner maximum is omitted if $b=2$. The positive number $\alpha_\lambda$ is rational by \cref{rmk:gamma-r-positive-rational}.

For $g\in S_n$, let $c(g)$ count all cycles, including fixed points, and denote by $\tau(g)=n-c(g)$ its transposition length. Define also, for $2\leqslant r < b$,
\[
	p_r = \begin{cases}
		2, & \text{$d$ even and $r$ even,} \\
		1, & \text{otherwise.}
	\end{cases}
\]
For $n\geqslant b+1$ and $rs\leqslant n$, a permutation of type $r^s$ belongs to $G_n$ whenever $p_r\mid s$.

\begin{lemma}
	\label{lemma:uniform-lower-certificate}
	There is a constant $\gamma_\lambda$ such that $D_n\geqslant \alpha_\lambda n -\gamma_\lambda$ for all $n\geqslant b+1$.
\end{lemma}
\begin{proof}
The group $G_n$ contains an even permutation $g_n$ with $\tau(g_n)\geqslant n-2$: an $n$-cycle when $n$ is odd, or an $(n-1)$-cycle when $n$ is even. Since every generator has transposition length $d$, one gets $D_n \geqslant L_n(g_n) \geqslant (n-2)/d$.

Fix $2\leqslant r<b$. Choose $s=p_r\lfloor n/(r p_r)\rfloor$; then cycle type $r^s$ is possible on $n$ letters, and $rs>n-r p_r$. When $s\geqslant 1$, \cref{lemma:sharp-bound} gives
\[
	D_n \geqslant L_n(r^s) \geqslant \frac{s}{\Gamma_r} \geqslant \frac{n}{r\Gamma_r} - \frac{p_r}{\Gamma_r}.
\]
If $s=0$, this bound still holds trivially. Thus one can take
\[
	\gamma_\lambda = \max\left( \frac{2}{d}, \max_{2\leqslant r < b} \frac{p_r}{\Gamma_r} \right),
\]
where the inner maximum is omitted in case $b=2$.
\end{proof}

\begin{lemma}
	\label{lemma:long-cycle-peeling}
	Let $g\in G_n$ contain a cycle of length $\ell \geqslant b$. There is a generator $f\in \cc_{\lambda,n}$, supported within that cycle, such that $gf^{-1}$ replaces it by one cycle of length $\ell-d$ and $d$ fixed points.
\end{lemma}
\begin{proof}
It can be assumed without loss of generality that $g$ is a single cycle of length $\ell$ ($f$ will commute with other cycles of $g$, so parity of $g$ is not an issue), and, after relabelling, $g=(1,2,\ldots,\ell)$. Split the first $b$ entries into consecutive blocks of lengths $\lambda_1,\ldots,\lambda_h$. Let $f$ be the product of the cycles obtained by reading each block in order, i.e.
\[
	f = (1,\ldots,\lambda_1)(\lambda_1+1,\ldots,\lambda_1+\lambda_2)\ldots(b-\lambda_h+1,\ldots,b).
\]
Note that $f\in \cc_{\lambda,n}$. Let $y_1,\ldots,y_h$ be the first labels of the blocks. Then
\[
	gf^{-1} = (y_1,y_2,\ldots,y_h,b+1,\ldots,\ell).
\]
This is a cycle of length $h+\ell-b = \ell-d$.
\end{proof}
\begin{remark}
	\label{rmk:consecutive-peeling}
	Each application shortens the surviving cycle by $d$, and another application is possible exactly when its length is at least $b$. Thus after
	\begin{equation}
		\label{eq:q-ell}
		\nu(\ell) = 1+\lfloor (\ell-b)/d \rfloor
	\end{equation}
	applications the cycle of length $\ell\geqslant b$ can be transformed into a cycle of length below $b$.
\end{remark}

Choose a positive integer $P$ divisible by
\[
	d, \qquad rK_r \ \text{for}\ 2\leqslant r<b, \qquad \text{the denominator of $\alpha_\lambda$.}
\]
The numbers $K_r,t_r,N_r$ are provided by \cref{lemma:sharp-bound}. Put
\[
	T=P/d, \quad q=\alpha_\lambda P.
\]
Then $T\leqslant q$, because $\alpha_\lambda\geqslant 1/d$, and
\[
	\frac{P t_r}{r K_r} = \frac{P}{r\Gamma_r} \leqslant q.
\]
\begin{lemma}
	\label{lemma:uniform-upper-recurrence}
	For all sufficiently large $n$
	\[
		D_n \leqslant D_{n-P} + q.
	\]
\end{lemma}
\begin{proof}
We prove this lemma by providing a procedure to obtain $P$ fixed points with at most $q$ multiplications by generators. Note that $G_n \cap S_{n-P} = G_{n-P}$ (provided $n-P\geqslant b+1$), so (up to relabelling) obtaining $P$ fixed points is the same as passing to $G_{n-P}$.

For a target $g\in G_n$, let $Q_\mathrm{long}(g)$ be the sum of $\nu(\ell)$ as defined by \cref{eq:q-ell} over its cycles of length at least $b$. Let $c_r(g)$ denote the number of $r$-cycles of $g$ (in particular, $c_1(g)$ is the number of fixed points). Consider the following three cases.

First, if $Q_\mathrm{long}(g)\geqslant T$, apply exactly $T$ steps provided by \cref{lemma:long-cycle-peeling}, distributed among the long cycles. This shortens these long cycles and introduces at least $dT=P$ new fixed points.

Second, if $c_1(g) \geqslant P$, $g$ is already in a copy of $G_{n-P}$.

Third, suppose that $c_r(g) \geqslant P/r$ for some $r\in\{2,\ldots,b-1\}$. Choose exactly $P/r$ of those cycles. Since $rK_r \mid P$, split them into $P/(rK_r)$ groups of $K_r$ cycles. By \cref{lemma:sharp-bound}, each group is the value of a word of $t_r$ generators, so these cycles can be removed by a total of $\frac{P}{rK_r} t_r \leqslant q$ generators. Each such word uses $N_r$ labels, but since the value of the word only moves the selected cycles and fixes all other letters, the condition $n\geqslant N_r$ suffices for each such realization.

We now show that for sufficiently large $n$, one of these three cases must occur. Assume the contrary:
\[
	c_1(g) \leqslant P-1, \quad c_r(g)\leqslant P/r-1\ \text{for}\ 2\leqslant r < b, \quad Q_\mathrm{long}(g) \leqslant T-1.
\]
Every original long cycle of length $\ell$ has $\ell - d\nu(\ell) \in [b-d,b-1]$, and hence $\ell \leqslant d\nu(\ell)+b-1$. Each such cycle contributes at least $1$ to $Q_\mathrm{long}(g)$, so the total support of long cycles is at most $(d+b-1)(T-1)$. Thus, summing the lengths of cycles of $g$, one gets
\[
	n \leqslant N_\mathrm{bad} = (P-1) + \sum_{r=2}^{b-1} (P-r) + (d+b-1)(T-1).
\]
This means that for every
\[
	n \geqslant \max\left( b+1+P, N_\mathrm{bad}+1, \max_{2\leqslant r < b} N_r \right)
\]
one of the cases occurs, and every reduction procedure works.

Thus it is proven that $L_n(g) \leqslant D_{n-P} + q$, and taking the maximum over all $g$ completes the proof.
\end{proof}

\begin{proof}[Proof of \cref{thm:diameters-affine-porc}]
Fix one residue $u\in\{0,\ldots,P-1\}$ and write $n=u+kP$. For all sufficiently large $k$, \cref{lemma:uniform-upper-recurrence} shows that
\[
	F_u(k) = D_{u+kP} - kq \quad \text{satisfies} \quad F_u(k+1)\leqslant F_u(k).
\]
Each $F_u(k)$ is an integer, and \cref{lemma:uniform-lower-certificate} gives the uniform lower bound
\[
	F_u(k) \geqslant \alpha_\lambda (u+kP) - \gamma_\lambda - kq = \alpha_\lambda u - \gamma_\lambda,
\]
because $q=\alpha_\lambda P$. An integer sequence that is eventually non-increasing and bounded below eventually becomes constant.

There are $P$ residue classes, so a single $N$ works for all of them and $D_{n+P} = D_n+q$ for $n\geqslant N$.
\end{proof}
\begin{remark}
	The above proof also establishes that the slope of the affine PORC $D_n$ equals $q/P = \alpha_\lambda$.
\end{remark}

\section{Unions of conjugacy classes}
\label{sec:conjugacy-classes-unions}
The argument proving \cref{thm:diameters-affine-porc} can be extended to finite fixed unions of conjugacy classes.
\begin{corollary}
	\label{cor:classes-unions}
	Fix non-trivial cycle types $\lambda^{(1)},\ldots,\lambda^{(v)}$, each independent of $n$. Then the diameters of Cayley graphs relative to the union of the conjugacy classes $\cc_{\lambda^{(i)},n}$ are eventually affine PORC. 
\end{corollary}
\begin{proof}
Let $J$ be the union of their constituent cycle lengths, let $\mvec^{(j)}$ be the resource vector of $\lambda^{(j)}$, and let $b_* = \max_j \sum_\ell \ell m^{(j)}_\ell$. Denote $d_j = \sum_\ell (\ell-1)m^{(j)}_\ell$ and $d_* = \max_j d_j$.

\cref{lemma:baseline_word,lemma:connected-word-defect,lemma:dickson,lemma:finite-domination} depend only on $J$ and the chosen target length $r$, so they give the same finite dictionary $\dict_r$ for this union. \Cref{eq:gamma-r} can be replaced with
\begin{equation}
	\label{eq:gamma-r-union}
	\Gamma_r^{\cup} = \max\left\{ \sum_{A\in\dict_r} s_A x_A \ \middle|\ x_A\geqslant 0,\ y_j\geqslant 0,\ \sum_{j=1}^v y_j=1,\ \sum_{A\in\dict_r} \cvec^A x_A = \sum_{j=1}^v y_j \mvec^{(j)} \right\}.
\end{equation}
Here $y_j$ is the proportion of generators of type $\lambda^{(j)}$ (hence $\sum_j y_j=1$). The right hand side of the last equality is bounded, so the feasibility domain is a non-empty compact rational polytope.

Both parts of \cref{lemma:sharp-bound} also extend. For a word of $t\geqslant 1$ generators, with $q_j$ of type $\lambda^{(j)}$, divide its resource vector by $t$ and take $y_j=q_j/t$. The same argument gives $s\leqslant t \Gamma_r^{\cup}$.  Conversely, clear the denominators of a maximizing solution, putting $q_j=Qy_j$. Each column in the cyclic-rotation construction then has resource vector $\sum_j q_j \mvec^{(j)}$. Its disjoint cycles can be partitioned into $q_j$ generators of type $\lambda^{(j)}$. We therefore obtain finite factorizations of permutations of cycle type $r^{K_r}$ in $N_r$ letters using $t_r$ generators, with $t_r/K_r = 1/\Gamma_r^{\cup}$.

Set
\begin{equation}
	\label{eq:common-slope-union}
	\alpha_* = \max\left( \frac{1}{d_*}, \max_{2\leqslant r < b_*} \frac{1}{r\Gamma_r^{\cup}} \right).
\end{equation}
In \cref{lemma:uniform-lower-certificate}, every generator now has transposition length at most $d_*$. For the short-cycle targets, take $p_r=2$ when every $d_j$ is even and $r$ is even, and $p_r=1$ otherwise. The same proof gives $D_n \geqslant \alpha_* n-\gamma_*$ for some constant $\gamma_*$.

For \cref{lemma:long-cycle-peeling}, choose a type $j$ with $d_j=d_*$. Since its support size is at most $b_*$, every cycle of length at least $b_*$ can be shortened by $d_*$ using one generator of this type.

The rest of the proof is the same as in the single conjugacy class case.
\end{proof}

Consider now the following general problem: given a pair of generating sequences $X_n, Y_n$ such that the corresponding diameter sequences $D^X_n,D^Y_n$ are quasipolynomial, is it true that the union $X_n\cup Y_n$ produces a quasipolynomial diameter sequence? It is rather easy to construct counter-examples. Below we outline some of them:
\begin{enumerate}
\item Consider a binary heap tree $B_n$ on $n$ vertices $1,\ldots,n$, where the vertex $i$ has $2i$ for its left child and $2i+1$ for its right child (when those labels are at most $n$). Consider the following orders on $\{1,\ldots,n\}$: $p_n$ traverses $B_n$ depth-first, visiting the left subtree first, and $q_n$ visits the right subtree first. Denote
\[
	X_n = \big\{ (p_n(i),p_n(i+1)) \mid i=1,\ldots,n-1 \big\}, \quad
	Y_n = \big\{ (q_n(i),q_n(i+1)) \mid i=1,\ldots,n-1 \big\}.
\]
Up to relabelling, those are just the usual Coxeter generators for $S_n$, so $D^X_n=D^Y_n = \binom{n}{2}$. However, the union $X_n\cup Y_n$ contains the set of transpositions along the edges of a binary heap tree, and it is known to have diameters $D^B_n = \Theta(n\log n)$. A lower bound of the same growth rate can be obtained by considering a related token swapping problem \cite{TokenSwappingOnTrees}.
\item Denote by $T_n$ the set of all transpositions in $S_n$, and for $m=\lfloor n/2\rfloor$ take $A,B\subset\{1,\ldots,n\}$ such that $|A|=|B|=m$ and $|A\setminus B|=|B\setminus A| = k$ for some prescribed $3\leqslant k\leqslant\lfloor m/2\rfloor$. Denote $X_n = T_n \cup S_A$ and $Y_n = T_n\cup S_B$, where $S_A,S_B<S_n$ are symmetric subgroups supported on $A$ and $B$. Then $D^X_n=D^Y_n = \lceil n/2\rceil+1$, while $D^{X\cup Y}_n = \lceil n/2\rceil -k(n)+3$, so choosing suitable $k(n)$ makes this sequence fail to be eventually quasipolynomial.
\item One can even make conjugation-invariant examples. Let
\[
	U_n = S_n \setminus \cc_{(2),n}, \quad V_n = S_n \setminus \cc_{(3),n}.
\]
One trivially has $D^U_n=D^V_n=2$ eventually. Now take $X_n=U_n$, and set $Y_n = U_n$ on squares and $V_n$ otherwise. Then $D^{X\cup Y}=2$ on squares and $1$ otherwise.
\end{enumerate}

\section{Explicit diameters for small $n$}
\label{sec:explicit-diameters}
For any given cycle type (or, more generally, a finite union of cycle types) one can calculate the exact diameter value for a given $n$ using character theory\footnote{The implementation can be found at \href{https://github.com/special-linear/conjugacy_bfs_characters}{\texttt{github.com/special-linear/conjugacy\_bfs\_characters}}.}. The data for some cases where the exact answer has not been established can be found in \cref{table:diameters,table:diameters-unions}.

The main observation suggested by these tables are:
\begin{itemize}
\item The computed values demonstrate that the affine PORC behavior often begins at a small degree, so the starting points of eventual quasipolynomials are significantly overestimated in \cref{thm:diameters-affine-porc,cor:classes-unions}. In practice the diameters usually appear to follow affine formulas from the first or second admissible degree.
\item In many cases  of interest it is easy to come up with very simple (conjectural) explicit formulas. In other cases the formulas appear to be manageable but less simple. For example:
\begin{itemize}
	\item All permutations supported on exactly $k$ letters:
	\[
		\begin{cases}
			\left\lceil \frac{n}{k-1} \right\rceil, & \text{$k$ odd,} \\[5pt]
			\left\lceil \frac{n}{k-1} \right\rceil + [n \equiv k-1 \pmod{2(k-1)}] & \text{$k$ even.}
		\end{cases}
	\]
	\item All cycles of length at most $k$:
	\[
		\begin{cases}
			\lceil (n-1)/3 \rceil + [n\equiv 4 \pmod{6}], & k = 5, \\
			\lceil n/4 \rceil + [n\equiv 4 \pmod{8}], & k = 6, \\
			\lceil (n-1)/4 \rceil, & k = 7, \\
			\lceil (n-1)/5 \rceil + [n\equiv 6 \pmod{10}], & k = 8.
		\end{cases}
	\]
	A general formula working for all $k$ remains unclear.
	\item Involutions supported on at most $2k$ letters, i.e. of types $2^m$, $m\leqslant k$: $\left\lceil \frac{n-1}{k} \right\rceil$.
\end{itemize}
\item It appears that for many cycle types the determination of the common slope in the corresponding affine PORC is much more direct than through \cref{eq:common-slope,eq:common-slope-union} and solving the corresponding optimization problem (\cref{eq:gamma-r,eq:gamma-r-union}) on dictionaries of choice (which are not immediate to construct). It appears, for example, that for the cycle type $2^a{\cdot}3^b{\cdot}4^c$ the common slope is $\frac{1}{q}$, where $q=\min\big(a+2b+3c,2(a+b+c)\big)$. For $2^a{\cdot}3^b$ this simplifies to $\frac{1}{a+2b}$.

However, general mixed types the common slope needs not be a reciprocal of an integer. For example, it appears to be equal to $2/9$ for $5{\cdot}2$ and to $3/16$ for $6{\cdot}2$.
\end{itemize}

\section*{Acknowledgements and AI use declaration}
The author is grateful to the CayleyPy community for many useful discussions on the subject.

The results of the present paper were obtained in close collaboration with ChatGPT, and it has also been used for a more extensive literature review and manuscript preparation and polishing. OpenAI's Codex was used extensively in the Lean formalization. The computational code used to obtain the results of \cref{sec:explicit-diameters} was written by Claude following the specification produced by a human-ChatGPT interaction.

\bibliographystyle{amsalpha}
\bibliography{conjugacy_diameter_porc.bib}

\clearpage

\begingroup
\small
\setlength{\tabcolsep}{1pt}
\setlength{\LTcapwidth}{\textwidth}
\renewcommand{\arraystretch}{1.08}
\begin{longtable}{@{}l|rrrrrrrrrrrrrrrrrrrrrrrrrrrrrrrrrrrrr@{}}
	\caption{Diameters of the identity component for fixed cycle types ($n=6$--$42$).}
	\label{table:diameters}\\
	\hline\hline
	C.t. & $\hphantom{00}\mathllap{6}$ & $\hphantom{00}\mathllap{7}$ & $\hphantom{00}\mathllap{8}$ & $\hphantom{00}\mathllap{9}$ & $\hphantom{00}\mathllap{10}$ & $\hphantom{00}\mathllap{11}$ & $\hphantom{00}\mathllap{12}$ & $\hphantom{00}\mathllap{13}$ & $\hphantom{00}\mathllap{14}$ & $\hphantom{00}\mathllap{15}$ & $\hphantom{00}\mathllap{16}$ & $\hphantom{00}\mathllap{17}$ & $\hphantom{00}\mathllap{18}$ & $\hphantom{00}\mathllap{19}$ & $\hphantom{00}\mathllap{20}$ & $\hphantom{00}\mathllap{21}$ & $\hphantom{00}\mathllap{22}$ & $\hphantom{00}\mathllap{23}$ & $\hphantom{00}\mathllap{24}$ & $\hphantom{00}\mathllap{25}$ & $\hphantom{00}\mathllap{26}$ & $\hphantom{00}\mathllap{27}$ & $\hphantom{00}\mathllap{28}$ & $\hphantom{00}\mathllap{29}$ & $\hphantom{00}\mathllap{30}$ & $\hphantom{00}\mathllap{31}$ & $\hphantom{00}\mathllap{32}$ & $\hphantom{00}\mathllap{33}$ & $\hphantom{00}\mathllap{34}$ & $\hphantom{00}\mathllap{35}$ & $\hphantom{00}\mathllap{36}$ & $\hphantom{00}\mathllap{37}$ & $\hphantom{00}\mathllap{38}$ & $\hphantom{00}\mathllap{39}$ & $\hphantom{00}\mathllap{40}$ & $\hphantom{00}\mathllap{41}$ & $\hphantom{00}\mathllap{42}$ \\
	\hline
	\endfirsthead
	\multicolumn{38}{c}{\tablename\ \thetable{} -- continued} \\
	\hline\hline
	C.t. & $\hphantom{00}\mathllap{6}$ & $\hphantom{00}\mathllap{7}$ & $\hphantom{00}\mathllap{8}$ & $\hphantom{00}\mathllap{9}$ & $\hphantom{00}\mathllap{10}$ & $\hphantom{00}\mathllap{11}$ & $\hphantom{00}\mathllap{12}$ & $\hphantom{00}\mathllap{13}$ & $\hphantom{00}\mathllap{14}$ & $\hphantom{00}\mathllap{15}$ & $\hphantom{00}\mathllap{16}$ & $\hphantom{00}\mathllap{17}$ & $\hphantom{00}\mathllap{18}$ & $\hphantom{00}\mathllap{19}$ & $\hphantom{00}\mathllap{20}$ & $\hphantom{00}\mathllap{21}$ & $\hphantom{00}\mathllap{22}$ & $\hphantom{00}\mathllap{23}$ & $\hphantom{00}\mathllap{24}$ & $\hphantom{00}\mathllap{25}$ & $\hphantom{00}\mathllap{26}$ & $\hphantom{00}\mathllap{27}$ & $\hphantom{00}\mathllap{28}$ & $\hphantom{00}\mathllap{29}$ & $\hphantom{00}\mathllap{30}$ & $\hphantom{00}\mathllap{31}$ & $\hphantom{00}\mathllap{32}$ & $\hphantom{00}\mathllap{33}$ & $\hphantom{00}\mathllap{34}$ & $\hphantom{00}\mathllap{35}$ & $\hphantom{00}\mathllap{36}$ & $\hphantom{00}\mathllap{37}$ & $\hphantom{00}\mathllap{38}$ & $\hphantom{00}\mathllap{39}$ & $\hphantom{00}\mathllap{40}$ & $\hphantom{00}\mathllap{41}$ & $\hphantom{00}\mathllap{42}$ \\
	\hline
	\endhead
	\hline
	\multicolumn{38}{r}{Continued on next page} \\
	\endfoot
	\hline\hline
	\endlastfoot
	$4$ & 3 & 4 & 4 & 5 & 5 & 6 & 6 & 7 & 7 & 8 & 8 & 9 & 9 & 10 & 10 & 11 & 11 & 12 & 12 & 13 & 13 & 14 & 14 & 15 & 15 & 16 & 16 & 17 & 17 & 18 & 18 & 19 & 19 & 20 & 20 & 21 & 21 \\
	$6$ & 3 & 3 & 3 & 3 & 4 & 4 & 4 & 5 & 5 & 5 & 5 & 5 & 6 & 6 & 6 & 7 & 7 & 7 & 7 & 7 & 8 & 8 & 8 & 9 & 9 & 9 & 9 & 9 & 10 & 10 & 10 & 11 & 11 & 11 & 11 & 11 & 12 \\
	$8$ & -- & -- & 3 & 3 & 3 & 3 & 4 & 4 & 4 & 4 & 4 & 5 & 5 & 5 & 5 & 5 & 6 & 6 & 6 & 6 & 7 & 7 & 7 & 7 & 7 & 7 & 8 & 8 & 8 & 8 & 8 & 9 & 9 & 9 & 9 & 9 & 10 \\
	$10$ & -- & -- & -- & -- & 3 & 3 & 3 & 3 & 3 & 4 & 4 & 4 & 4 & 4 & 4 & 5 & 5 & 5 & 5 & 5 & 5 & 6 & 6 & 6 & 6 & 6 & 6 & 7 & 7 & 7 & 7 & 7 & 7 & 8 & 8 & 8 & 8 \\
	$3{\cdot}2$ & 3 & 3 & 3 & 4 & 4 & 4 & 5 & 5 & 5 & 6 & 6 & 6 & 7 & 7 & 7 & 8 & 8 & 8 & 9 & 9 & 9 & 10 & 10 & 10 & 11 & 11 & 11 & 12 & 12 & 12 & 13 & 13 & 13 & 14 & 14 & 14 & 15 \\
	$3{\cdot}2^{2}$ & -- & 2 & 2 & 2 & 2 & 3 & 3 & 3 & 3 & 4 & 4 & 4 & 4 & 5 & 5 & 5 & 5 & 6 & 6 & 6 & 6 & 7 & 7 & 7 & 7 & 8 & 8 & 8 & 8 & 9 & 9 & 9 & 9 & 10 & 10 & 10 & 10 \\
	$3{\cdot}2^{3}$ & -- & -- & -- & 3 & 3 & 3 & 3 & 4 & 4 & 4 & 4 & 4 & 5 & 5 & 5 & 5 & 5 & 6 & 6 & 6 & 6 & 6 & 7 & 7 & 7 & 7 & 7 & 8 & 8 & 8 & 8 & 8 & 9 & 9 & 9 & 9 & 9 \\
	$3{\cdot}2^{4}$ & -- & -- & -- & -- & -- & 2 & 2 & 2 & 2 & 3 & 3 & 3 & 3 & 3 & 3 & 4 & 4 & 4 & 4 & 4 & 4 & 5 & 5 & 5 & 5 & 5 & 5 & 6 & 6 & 6 & 6 & 6 & 6 & 7 & 7 & 7 & 7 \\
	$3{\cdot}2^{5}$ & -- & -- & -- & -- & -- & -- & -- & 3 & 3 & 3 & 3 & 4 & 4 & 4 & 4 & 4 & 4 & 4 & 5 & 5 & 5 & 5 & 5 & 5 & 5 & 6 & 6 & 6 & 6 & 6 & 6 & 6 & 7 & 7 & 7 & 7 & 7 \\
	$3^{2}$ & 3 & 2 & 2 & 2 & 3 & 3 & 3 & 3 & 4 & 4 & 4 & 4 & 5 & 5 & 5 & 5 & 6 & 6 & 6 & 6 & 7 & 7 & 7 & 7 & 8 & 8 & 8 & 8 & 9 & 9 & 9 & 9 & 10 & 10 & 10 & 10 & 11 \\
	$3^{2}{\cdot}2$ & -- & -- & 3 & 3 & 3 & 3 & 3 & 4 & 4 & 4 & 4 & 4 & 5 & 5 & 5 & 5 & 5 & 6 & 6 & 6 & 6 & 6 & 7 & 7 & 7 & 7 & 7 & 8 & 8 & 8 & 8 & 8 & 9 & 9 & 9 & 9 & 9 \\
	$3^{2}{\cdot}2^{2}$ & -- & -- & -- & -- & 2 & 2 & 2 & 2 & 2 & 3 & 3 & 3 & 3 & 3 & 3 & 4 & 4 & 4 & 4 & 4 & 4 & 5 & 5 & 5 & 5 & 5 & 5 & 6 & 6 & 6 & 6 & 6 & 6 & 7 & 7 & 7 & 7 \\
	$3^{2}{\cdot}2^{3}$ & -- & -- & -- & -- & -- & -- & 3 & 3 & 3 & 3 & 3 & 4 & 4 & 4 & 4 & 4 & 4 & 4 & 5 & 5 & 5 & 5 & 5 & 5 & 5 & 6 & 6 & 6 & 6 & 6 & 6 & 6 & 7 & 7 & 7 & 7 & 7 \\
	$3^{2}{\cdot}2^{4}$ & -- & -- & -- & -- & -- & -- & -- & -- & 2 & 2 & 2 & 2 & 2 & 3 & 3 & 3 & 3 & 3 & 3 & 3 & 3 & 4 & 4 & 4 & 4 & 4 & 4 & 4 & 4 & 5 & 5 & 5 & 5 & 5 & 5 & 5 & 5 \\
	$3^{2}{\cdot}2^{5}$ & -- & -- & -- & -- & -- & -- & -- & -- & -- & -- & 3 & 3 & 3 & 3 & 3 & 4 & 4 & 4 & 4 & 4 & 4 & 4 & 4 & 4 & 5 & 5 & 5 & 5 & 5 & 5 & 5 & 5 & 5 & 6 & 6 & 6 & 6 \\
	$3^{3}$ & -- & -- & -- & 3 & 3 & 2 & 2 & 2 & 3 & 3 & 3 & 3 & 3 & 3 & 4 & 4 & 4 & 4 & 4 & 4 & 5 & 5 & 5 & 5 & 5 & 5 & 6 & 6 & 6 & 6 & 6 & 6 & 7 & 7 & 7 & 7 & 7 \\
	$3^{3}{\cdot}2$ & -- & -- & -- & -- & -- & 3 & 3 & 3 & 3 & 3 & 3 & 4 & 4 & 4 & 4 & 4 & 4 & 4 & 5 & 5 & 5 & 5 & 5 & 5 & 5 & 6 & 6 & 6 & 6 & 6 & 6 & 6 & 7 & 7 & 7 & 7 & 7 \\
	$3^{3}{\cdot}2^{2}$ & -- & -- & -- & -- & -- & -- & -- & 2 & 2 & 2 & 2 & 2 & 2 & 3 & 3 & 3 & 3 & 3 & 3 & 3 & 3 & 4 & 4 & 4 & 4 & 4 & 4 & 4 & 4 & 5 & 5 & 5 & 5 & 5 & 5 & 5 & 5 \\
	$3^{3}{\cdot}2^{3}$ & -- & -- & -- & -- & -- & -- & -- & -- & -- & 3 & 3 & 3 & 3 & 3 & 3 & 4 & 4 & 4 & 4 & 4 & 4 & 4 & 4 & 4 & 5 & 5 & 5 & 5 & 5 & 5 & 5 & 5 & 5 & 6 & 6 & 6 & 6 \\
	$4{\cdot}2$ & 2 & 2 & 2 & 2 & 2 & 3 & 3 & 3 & 3 & 4 & 4 & 4 & 4 & 5 & 5 & 5 & 5 & 6 & 6 & 6 & 6 & 7 & 7 & 7 & 7 & 8 & 8 & 8 & 8 & 9 & 9 & 9 & 9 & 10 & 10 & 10 & 10 \\
	$4{\cdot}2^{2}$ & -- & -- & 3 & 3 & 3 & 3 & 3 & 4 & 4 & 4 & 4 & 4 & 5 & 5 & 5 & 5 & 5 & 6 & 6 & 6 & 6 & 6 & 7 & 7 & 7 & 7 & 7 & 8 & 8 & 8 & 8 & 8 & 9 & 9 & 9 & 9 & 9 \\
	$4{\cdot}2^{3}$ & -- & -- & -- & -- & 2 & 2 & 2 & 2 & 2 & 3 & 3 & 3 & 3 & 3 & 3 & 4 & 4 & 4 & 4 & 4 & 4 & 5 & 5 & 5 & 5 & 5 & 5 & 6 & 6 & 6 & 6 & 6 & 6 & 7 & 7 & 7 & 7 \\
	$4{\cdot}2^{5}$ & -- & -- & -- & -- & -- & -- & -- & -- & 2 & 2 & 2 & 2 & 2 & 3 & 3 & 3 & 3 & 3 & 3 & 3 & 3 & 4 & 4 & 4 & 4 & 4 & 4 & 4 & 4 & 5 & 5 & 5 & 5 & 5 & 5 & 5 & 5 \\
	$4{\cdot}3{\cdot}2$ & -- & -- & -- & 2 & 2 & 2 & 2 & 2 & 2 & 3 & 3 & 3 & 3 & 3 & 4 & 4 & 4 & 4 & 4 & 4 & 4 & 5 & 5 & 5 & 5 & 5 & 6 & 6 & 6 & 6 & 6 & 6 & 6 & 7 & 7 & 7 & 7 \\
	$4{\cdot}3{\cdot}2^{2}$ & -- & -- & -- & -- & -- & 3 & 3 & 3 & 3 & 3 & 3 & 4 & 4 & 4 & 4 & 4 & 4 & 4 & 5 & 5 & 5 & 5 & 5 & 5 & 5 & 6 & 6 & 6 & 6 & 6 & 6 & 6 & 7 & 7 & 7 & 7 & 7 \\
	$4{\cdot}3{\cdot}2^{3}$ & -- & -- & -- & -- & -- & -- & -- & 2 & 2 & 2 & 2 & 2 & 2 & 3 & 3 & 3 & 3 & 3 & 3 & 3 & 3 & 4 & 4 & 4 & 4 & 4 & 4 & 4 & 4 & 5 & 5 & 5 & 5 & 5 & 5 & 5 & 5 \\
	$4{\cdot}3{\cdot}2^{4}$ & -- & -- & -- & -- & -- & -- & -- & -- & -- & 3 & 3 & 3 & 3 & 3 & 3 & 4 & 4 & 4 & 4 & 4 & 4 & 4 & 4 & 4 & 5 & 5 & 5 & 5 & 5 & 5 & 5 & 5 & 5 & 6 & 6 & 6 & 6 \\
	$5{\cdot}2$ & -- & 3 & 3 & 3 & 3 & 3 & 4 & 4 & 4 & 4 & 5 & 5 & 5 & 5 & 5 & 6 & 6 & 6 & 6 & 7 & 7 & 7 & 7 & 7 & 8 & 8 & 8 & 8 & 9 & 9 & 9 & 9 & 9 & 10 & 10 & 10 & 10 \\
	$5{\cdot}2^{2}$ & -- & -- & -- & 2 & 2 & 2 & 2 & 2 & 2 & 3 & 3 & 3 & 3 & 3 & 3 & 4 & 4 & 4 & 4 & 4 & 4 & 5 & 5 & 5 & 5 & 5 & 5 & 6 & 6 & 6 & 6 & 6 & 6 & 7 & 7 & 7 & 7 \\
	$5{\cdot}2^{3}$ & -- & -- & -- & -- & -- & 3 & 3 & 3 & 3 & 3 & 3 & 4 & 4 & 4 & 4 & 4 & 4 & 4 & 5 & 5 & 5 & 5 & 5 & 5 & 5 & 6 & 6 & 6 & 6 & 6 & 6 & 6 & 7 & 7 & 7 & 7 & 7 \\
	$5{\cdot}2^{4}$ & -- & -- & -- & -- & -- & -- & -- & 2 & 2 & 2 & 2 & 2 & 2 & 3 & 3 & 3 & 3 & 3 & 3 & 3 & 3 & 4 & 4 & 4 & 4 & 4 & 4 & 4 & 4 & 5 & 5 & 5 & 5 & 5 & 5 & 5 & 5 \\
	$5{\cdot}2^{5}$ & -- & -- & -- & -- & -- & -- & -- & -- & -- & 3 & 3 & 3 & 3 & 3 & 3 & 4 & 4 & 4 & 4 & 4 & 4 & 4 & 4 & 4 & 5 & 5 & 5 & 5 & 5 & 5 & 5 & 5 & 5 & 6 & 6 & 6 & 6 \\
	$5{\cdot}3$ & -- & -- & 2 & 2 & 2 & 2 & 3 & 3 & 3 & 3 & 4 & 4 & 4 & 4 & 4 & 4 & 5 & 5 & 5 & 5 & 5 & 6 & 6 & 6 & 6 & 6 & 7 & 7 & 7 & 7 & 8 & 8 & 8 & 8 & 8 & 8 & 9 \\
	$5{\cdot}3{\cdot}2$ & -- & -- & -- & -- & 3 & 3 & 3 & 3 & 3 & 3 & 4 & 4 & 4 & 4 & 4 & 4 & 5 & 5 & 5 & 5 & 5 & 5 & 5 & 5 & 6 & 6 & 6 & 6 & 6 & 6 & 6 & 7 & 7 & 7 & 7 & 7 & 7 \\
	$5{\cdot}4$ & -- & -- & -- & 3 & 3 & 3 & 3 & 3 & 4 & 4 & 4 & 4 & 4 & 4 & 5 & 5 & 5 & 5 & 5 & 5 & 6 & 6 & 6 & 6 & 6 & 6 & 7 & 7 & 7 & 7 & 7 & 7 & 8 & 8 & 8 & 8 & 8 \\
	$5{\cdot}4{\cdot}2$ & -- & -- & -- & -- & -- & 2 & 2 & 2 & 2 & 2 & 2 & 2 & 3 & 3 & 3 & 3 & 3 & 3 & 4 & 4 & 4 & 4 & 4 & 4 & 4 & 4 & 4 & 5 & 5 & 5 & 5 & 5 & 5 & 6 & 6 & 6 & 6 \\
	$5{\cdot}4{\cdot}3$ & -- & -- & -- & -- & -- & -- & 3 & 3 & 3 & 3 & 3 & 3 & 4 & 4 & 4 & 4 & 4 & 4 & 4 & 4 & 5 & 5 & 5 & 5 & 5 & 5 & 5 & 5 & 6 & 6 & 6 & 6 & 6 & 6 & 6 & 6 & 7 \\
	$5{\cdot}4{\cdot}3{\cdot}2$ & -- & -- & -- & -- & -- & -- & -- & -- & 2 & 2 & 2 & 2 & 2 & 2 & 2 & 2 & 3 & 3 & 3 & 3 & 3 & 3 & 3 & 3 & 3 & 3 & 4 & 4 & 4 & 4 & 4 & 4 & 4 & 4 & 4 & 4 & 5 \\
\end{longtable}

\endgroup

\clearpage

\begingroup
\footnotesize
\setlength{\tabcolsep}{1pt}
\setlength{\LTcapwidth}{\textwidth}
\renewcommand{\arraystretch}{1.08}
\begin{longtable}{@{}l|rrrrrrrrrrrrrrrrrrrrrrrrrrrrrrrrrrrrr@{}}
	\caption{Diameters of the identity component for unions of conjugacy classes ($n=6$--$42$).}
	\label{table:diameters-unions}\\
	\hline\hline
	C.t. & $\hphantom{00}\mathllap{6}$ & $\hphantom{00}\mathllap{7}$ & $\hphantom{00}\mathllap{8}$ & $\hphantom{00}\mathllap{9}$ & $\hphantom{00}\mathllap{10}$ & $\hphantom{00}\mathllap{11}$ & $\hphantom{00}\mathllap{12}$ & $\hphantom{00}\mathllap{13}$ & $\hphantom{00}\mathllap{14}$ & $\hphantom{00}\mathllap{15}$ & $\hphantom{00}\mathllap{16}$ & $\hphantom{00}\mathllap{17}$ & $\hphantom{00}\mathllap{18}$ & $\hphantom{00}\mathllap{19}$ & $\hphantom{00}\mathllap{20}$ & $\hphantom{00}\mathllap{21}$ & $\hphantom{00}\mathllap{22}$ & $\hphantom{00}\mathllap{23}$ & $\hphantom{00}\mathllap{24}$ & $\hphantom{00}\mathllap{25}$ & $\hphantom{00}\mathllap{26}$ & $\hphantom{00}\mathllap{27}$ & $\hphantom{00}\mathllap{28}$ & $\hphantom{00}\mathllap{29}$ & $\hphantom{00}\mathllap{30}$ & $\hphantom{00}\mathllap{31}$ & $\hphantom{00}\mathllap{32}$ & $\hphantom{00}\mathllap{33}$ & $\hphantom{00}\mathllap{34}$ & $\hphantom{00}\mathllap{35}$ & $\hphantom{00}\mathllap{36}$ & $\hphantom{00}\mathllap{37}$ & $\hphantom{00}\mathllap{38}$ & $\hphantom{00}\mathllap{39}$ & $\hphantom{00}\mathllap{40}$ & $\hphantom{00}\mathllap{41}$ & $\hphantom{00}\mathllap{42}$ \\
	\hline
	\endfirsthead
	\multicolumn{38}{c}{\tablename\ \thetable{} -- continued} \\
	\hline\hline
	C.t. & $\hphantom{00}\mathllap{6}$ & $\hphantom{00}\mathllap{7}$ & $\hphantom{00}\mathllap{8}$ & $\hphantom{00}\mathllap{9}$ & $\hphantom{00}\mathllap{10}$ & $\hphantom{00}\mathllap{11}$ & $\hphantom{00}\mathllap{12}$ & $\hphantom{00}\mathllap{13}$ & $\hphantom{00}\mathllap{14}$ & $\hphantom{00}\mathllap{15}$ & $\hphantom{00}\mathllap{16}$ & $\hphantom{00}\mathllap{17}$ & $\hphantom{00}\mathllap{18}$ & $\hphantom{00}\mathllap{19}$ & $\hphantom{00}\mathllap{20}$ & $\hphantom{00}\mathllap{21}$ & $\hphantom{00}\mathllap{22}$ & $\hphantom{00}\mathllap{23}$ & $\hphantom{00}\mathllap{24}$ & $\hphantom{00}\mathllap{25}$ & $\hphantom{00}\mathllap{26}$ & $\hphantom{00}\mathllap{27}$ & $\hphantom{00}\mathllap{28}$ & $\hphantom{00}\mathllap{29}$ & $\hphantom{00}\mathllap{30}$ & $\hphantom{00}\mathllap{31}$ & $\hphantom{00}\mathllap{32}$ & $\hphantom{00}\mathllap{33}$ & $\hphantom{00}\mathllap{34}$ & $\hphantom{00}\mathllap{35}$ & $\hphantom{00}\mathllap{36}$ & $\hphantom{00}\mathllap{37}$ & $\hphantom{00}\mathllap{38}$ & $\hphantom{00}\mathllap{39}$ & $\hphantom{00}\mathllap{40}$ & $\hphantom{00}\mathllap{41}$ & $\hphantom{00}\mathllap{42}$ \\
	\hline
	\endhead
	\hline
	\multicolumn{38}{r}{Continued on next page} \\
	\endfoot
	\hline\hline
	\endlastfoot
	supp $\leqslant 3$ & 3 & 3 & 4 & 4 & 5 & 5 & 6 & 6 & 7 & 7 & 8 & 8 & 9 & 9 & 10 & 10 & 11 & 11 & 12 & 12 & 13 & 13 & 14 & 14 & 15 & 15 & 16 & 16 & 17 & 17 & 18 & 18 & 19 & 19 & 20 & 20 & 21 \\
	supp $\leqslant 4$ & 2 & 2 & 3 & 3 & 3 & 4 & 4 & 4 & 5 & 5 & 5 & 6 & 6 & 6 & 7 & 7 & 7 & 8 & 8 & 8 & 9 & 9 & 9 & 10 & 10 & 10 & 11 & 11 & 11 & 12 & 12 & 12 & 13 & 13 & 13 & 14 & 14 \\
	supp $\leqslant 5$ & 2 & 2 & 2 & 2 & 3 & 3 & 3 & 3 & 4 & 4 & 4 & 4 & 5 & 5 & 5 & 5 & 6 & 6 & 6 & 6 & 7 & 7 & 7 & 7 & 8 & 8 & 8 & 8 & 9 & 9 & 9 & 9 & 10 & 10 & 10 & 10 & 11 \\
	supp $\leqslant 6$ & 1 & 2 & 2 & 2 & 2 & 2 & 3 & 3 & 3 & 3 & 3 & 4 & 4 & 4 & 4 & 4 & 5 & 5 & 5 & 5 & 5 & 6 & 6 & 6 & 6 & 6 & 7 & 7 & 7 & 7 & 7 & 8 & 8 & 8 & 8 & 8 & 9 \\
	supp $\leqslant 7$ & -- & 1 & 2 & 2 & 2 & 2 & 2 & 2 & 3 & 3 & 3 & 3 & 3 & 3 & 4 & 4 & 4 & 4 & 4 & 4 & 5 & 5 & 5 & 5 & 5 & 5 & 6 & 6 & 6 & 6 & 6 & 6 & 7 & 7 & 7 & 7 & 7 \\
	supp $\leqslant 8$ & -- & -- & 1 & 2 & 2 & 2 & 2 & 2 & 2 & 2 & 3 & 3 & 3 & 3 & 3 & 3 & 3 & 4 & 4 & 4 & 4 & 4 & 4 & 4 & 5 & 5 & 5 & 5 & 5 & 5 & 5 & 6 & 6 & 6 & 6 & 6 & 6 \\
	supp $= 4$ & 2 & 3 & 3 & 3 & 4 & 4 & 4 & 5 & 5 & 5 & 6 & 6 & 6 & 7 & 7 & 7 & 8 & 8 & 8 & 9 & 9 & 9 & 10 & 10 & 10 & 11 & 11 & 11 & 12 & 12 & 12 & 13 & 13 & 13 & 14 & 14 & 14 \\
	supp $= 5$ & 2 & 2 & 2 & 3 & 3 & 3 & 4 & 4 & 4 & 4 & 4 & 5 & 5 & 5 & 6 & 6 & 6 & 6 & 6 & 7 & 7 & 7 & 8 & 8 & 8 & 8 & 8 & 9 & 9 & 9 & 10 & 10 & 10 & 10 & 10 & 11 & 11 \\
	supp $= 6$ & 2 & 2 & 2 & 2 & 2 & 3 & 3 & 3 & 3 & 3 & 4 & 4 & 4 & 4 & 4 & 5 & 5 & 5 & 5 & 5 & 6 & 6 & 6 & 6 & 6 & 7 & 7 & 7 & 7 & 7 & 8 & 8 & 8 & 8 & 8 & 9 & 9 \\
	supp $= 7$ & -- & 2 & 2 & 2 & 2 & 2 & 2 & 3 & 3 & 3 & 3 & 3 & 4 & 4 & 4 & 4 & 4 & 4 & 4 & 5 & 5 & 5 & 5 & 5 & 6 & 6 & 6 & 6 & 6 & 6 & 6 & 7 & 7 & 7 & 7 & 7 & 8 \\
	supp $= 8$ & -- & -- & 2 & 2 & 2 & 2 & 2 & 2 & 2 & 3 & 3 & 3 & 3 & 3 & 3 & 3 & 4 & 4 & 4 & 4 & 4 & 4 & 4 & 5 & 5 & 5 & 5 & 5 & 5 & 5 & 6 & 6 & 6 & 6 & 6 & 6 & 6 \\
	cycles $\leqslant 4$ & 3 & 3 & 4 & 4 & 5 & 5 & 6 & 6 & 7 & 7 & 8 & 8 & 9 & 9 & 10 & 10 & 11 & 11 & 12 & 12 & 13 & 13 & 14 & 14 & 15 & 15 & 16 & 16 & 17 & 17 & 18 & 18 & 19 & 19 & 20 & 20 & 21 \\
	cycles $\leqslant 5$ & 2 & 2 & 3 & 3 & 4 & 4 & 4 & 4 & 5 & 5 & 6 & 6 & 6 & 6 & 7 & 7 & 8 & 8 & 8 & 8 & 9 & 9 & 10 & 10 & 10 & 10 & 11 & 11 & 12 & 12 & 12 & 12 & 13 & 13 & 14 & 14 & 14 \\
	cycles $\leqslant 6$ & 2 & 2 & 2 & 3 & 3 & 3 & 4 & 4 & 4 & 4 & 4 & 5 & 5 & 5 & 6 & 6 & 6 & 6 & 6 & 7 & 7 & 7 & 8 & 8 & 8 & 8 & 8 & 9 & 9 & 9 & 10 & 10 & 10 & 10 & 10 & 11 & 11 \\
	cycles $\leqslant 7$ & -- & 2 & 2 & 2 & 3 & 3 & 3 & 3 & 4 & 4 & 4 & 4 & 5 & 5 & 5 & 5 & 6 & 6 & 6 & 6 & 7 & 7 & 7 & 7 & 8 & 8 & 8 & 8 & 9 & 9 & 9 & 9 & 10 & 10 & 10 & 10 & 11 \\
	cycles $\leqslant 8$ & -- & -- & 2 & 2 & 2 & 2 & 3 & 3 & 3 & 3 & 4 & 4 & 4 & 4 & 4 & 4 & 5 & 5 & 5 & 5 & 6 & 6 & 6 & 6 & 6 & 6 & 7 & 7 & 7 & 7 & 8 & 8 & 8 & 8 & 8 & 8 & 9 \\
	invol. $\leqslant 4$ & 3 & 3 & 4 & 4 & 5 & 5 & 6 & 6 & 7 & 7 & 8 & 8 & 9 & 9 & 10 & 10 & 11 & 11 & 12 & 12 & 13 & 13 & 14 & 14 & 15 & 15 & 16 & 16 & 17 & 17 & 18 & 18 & 19 & 19 & 20 & 20 & 21 \\
	invol. $\leqslant 6$ & 2 & 2 & 3 & 3 & 3 & 4 & 4 & 4 & 5 & 5 & 5 & 6 & 6 & 6 & 7 & 7 & 7 & 8 & 8 & 8 & 9 & 9 & 9 & 10 & 10 & 10 & 11 & 11 & 11 & 12 & 12 & 12 & 13 & 13 & 13 & 14 & 14 \\
	invol. $\leqslant 8$ & -- & -- & 2 & 2 & 3 & 3 & 3 & 3 & 4 & 4 & 4 & 4 & 5 & 5 & 5 & 5 & 6 & 6 & 6 & 6 & 7 & 7 & 7 & 7 & 8 & 8 & 8 & 8 & 9 & 9 & 9 & 9 & 10 & 10 & 10 & 10 & 11 \\
	invol. $\leqslant 10$ & -- & -- & -- & -- & 2 & 2 & 3 & 3 & 3 & 3 & 3 & 4 & 4 & 4 & 4 & 4 & 5 & 5 & 5 & 5 & 5 & 6 & 6 & 6 & 6 & 6 & 7 & 7 & 7 & 7 & 7 & 8 & 8 & 8 & 8 & 8 & 9 \\
	invol. $\leqslant 12$ & -- & -- & -- & -- & -- & -- & 2 & 2 & 3 & 3 & 3 & 3 & 3 & 3 & 4 & 4 & 4 & 4 & 4 & 4 & 5 & 5 & 5 & 5 & 5 & 5 & 6 & 6 & 6 & 6 & 6 & 6 & 7 & 7 & 7 & 7 & 7 \\
	even cyc. $\leqslant 4$ & 3 & 4 & 4 & 5 & 5 & 6 & 6 & 7 & 7 & 8 & 8 & 9 & 9 & 10 & 10 & 11 & 11 & 12 & 12 & 13 & 13 & 14 & 14 & 15 & 15 & 16 & 16 & 17 & 17 & 18 & 18 & 19 & 19 & 20 & 20 & 21 & 21 \\
	even cyc. $\leqslant 6$ & 3 & 3 & 3 & 3 & 4 & 4 & 4 & 5 & 5 & 5 & 5 & 5 & 6 & 6 & 6 & 7 & 7 & 7 & 7 & 7 & 8 & 8 & 8 & 9 & 9 & 9 & 9 & 9 & 10 & 10 & 10 & 11 & 11 & 11 & 11 & 11 & 12 \\
	even cyc. $\leqslant 8$ & -- & -- & 3 & 3 & 3 & 3 & 4 & 4 & 4 & 4 & 4 & 5 & 5 & 5 & 5 & 5 & 6 & 6 & 6 & 6 & 7 & 7 & 7 & 7 & 7 & 7 & 8 & 8 & 8 & 8 & 8 & 9 & 9 & 9 & 9 & 9 & 10 \\
	odd cyc. $\leqslant 5$ & 2 & 2 & 3 & 3 & 3 & 4 & 4 & 4 & 5 & 5 & 6 & 6 & 6 & 6 & 7 & 7 & 7 & 8 & 8 & 8 & 9 & 9 & 10 & 10 & 10 & 10 & 11 & 11 & 11 & 12 & 12 & 12 & 13 & 13 & 14 & 14 & 14 \\
	odd cyc. $\leqslant 7$ & -- & 2 & 2 & 2 & 2 & 3 & 3 & 3 & 3 & 4 & 4 & 4 & 4 & 5 & 5 & 5 & 5 & 6 & 6 & 6 & 6 & 7 & 7 & 7 & 7 & 8 & 8 & 8 & 8 & 9 & 9 & 9 & 9 & 10 & 10 & 10 & 10 \\
	odd cyc. $\leqslant 9$ & -- & -- & -- & 2 & 2 & 2 & 2 & 2 & 3 & 3 & 3 & 3 & 3 & 4 & 4 & 4 & 4 & 4 & 4 & 4 & 5 & 5 & 5 & 5 & 5 & 6 & 6 & 6 & 6 & 6 & 6 & 6 & 7 & 7 & 7 & 7 & 7 \\
\end{longtable}

\endgroup

\end{document}